\documentclass[a4paper,11pt]{amsart}

\usepackage{amsthm}
\usepackage[arrow,curve,matrix]{xy}
\usepackage{amsmath}  
\usepackage{amssymb}
\usepackage{graphicx}
\usepackage{bbm}
\usepackage{mathrsfs}
\usepackage{amscd}
\usepackage[utf8x]{inputenc}
\usepackage{lscape}
\usepackage{a4wide}

\renewcommand{\mathbbm}{\mathbb}

\allowdisplaybreaks

\DeclareMathOperator{\pr}{pr}

\DeclareMathOperator{\id}{id}

\DeclareMathOperator{\Hom}{Hom}

\DeclareMathOperator{\Coh}{Coh}

\DeclareMathOperator{\D}{D}
\DeclareMathOperator{\K}{K}

\DeclareMathOperator{\Hilb}{Hilb}

\DeclareMathOperator{\glob}{\Gamma}

\DeclareMathOperator{\triv}{\mathsf{triv}}

\DeclareMathOperator{\Inf}{\mathsf{Inf}}
\DeclareMathOperator{\Res}{\mathsf{Res}}

\newcommand{\Fock}{\mathbb F}
\newcommand{\sym}{\mathfrak S}

\newcommand{\C}{\mathbbm C}
\newcommand{\N}{\mathbbm N}
\newcommand{\Q}{\mathbbm Q}

\newcommand{\ID}{\mathbbm D}
\newcommand{\Ik}{\mathsf k}

\newcommand{\Z}{\mathbbm Z}

\newcommand{\cE}{\mathcal E}

\newcommand{\cX}{\mathcal X}

\newcommand{\cR}{\mathcal R}
\newcommand{\cK}{\mathcal K}
\newcommand{\Heis}{\mathsf H}

\newcommand{\uk}{\underline{k}}
\newcommand{\kcomp}{\overline{k}}
\newcommand{\n}{\underline{n}}
\newcommand{\m}{\underline{m}}

\newcommand{\ncomp}{\overline{n}}
\newcommand{\mcomp}{\overline{m}}
\newcommand{\mncomp}{\overline{n+m}}
\newcommand{\Icomp}{\overline{I}}

\newcommand{\cT}{\mathcal T}

\newcommand{\alt}{\mathfrak a}

\newcommand{\reg}{\mathcal O}

\renewcommand{\theta}{\vartheta}
\renewcommand{\rho}{\varrho}
\renewcommand{\phi}{\varphi}
\renewcommand{\_}{\underline{\,\,\,\,}}

\newtheorem{theorem}{Theorem}[section]

\newtheorem{lemma}[theorem]{Lemma}
\newtheorem{cor}[theorem]{Corollary}

\theoremstyle{definition}
\newtheorem{defin}[theorem]{Definition}

\newtheorem{remark}[theorem]{Remark}

\usepackage{mathtools}

\begin{document}
\title{Symmetric quotient stacks and Heisenberg actions}
\author{Andreas Krug}
\address{Mathematics Institute, University of Warwick\\Zeeman Building, CV4 7AL\\Coventry, UK}
\email{a.krug@warwick.ac.uk}
\begin{abstract}
For every smooth projective variety $X$, we construct an action of the Heisenberg algebra on the direct sum of the Grothendieck groups of 
all the symmetric quotient stacks $[X^n/\sym_n]$ which contains the Fock space as a subrepresentation. The action is induced by functors on the level of the derived categories which form a weak categorification of the action. 
\end{abstract}
\maketitle
\section{Introduction}
A celebrated theorem of Nakajima \cite{Nak} and Grojnowski \cite{Groj} identifies the cohomology of Hilbert schemes of points on a surface with the Fock space representation of the Heisenberg Lie algebra associated to the cohomology of the surface itself.

One would like to lift the Heisenberg action to the level of the Grothendieck groups or, even better, the derived categories of the Hilbert schemes. Schiffmann and Vasserot \cite{SV} as well as Feigin and Tsymbaliuk \cite{FT} constructed a Heisenberg action on the Grothendieck groups in the case of the affine plane. In \cite{CL}, Cautis and Licata constructed a categorical Heisenberg action on the derived categories of the Hilbert schemes in the case that the surface is a minimal resolution of a Kleinian singularity. Their construction makes use of the derived McKay correspondence between the derived category  of the Hilbert scheme of points on a smooth quasi-projective surface and the derived category  of the symmetric quotient stack associated to the surface; see \cite{BKR} and \cite{Hai}. 

In this paper, we generalise the construction of \cite{CL}     
to obtain functors between the derived categories of the symmetric quotient stacks of an arbitrary smooth complete variety which descend to a Heisenberg action on the Grothendieck groups of the symmetric quotient stacks. Our construction can also be seen as a global version of a construction of Khovanov \cite{Khov} which deals with the case that the variety is a point. Note, however, that there are deeper categorical structures in \cite{CL} as well as \cite{Khov} that we do not generalise; see also Section \ref{open}.

Our construction is much closer to the construction of \cite{Groj} (see also \cite[Ch.\ 9]{Nakbook}) than to that of \cite{Nak}. In some sense, what our paper does is to fill in the proof of the claim made in \cite[footnote 3]{Groj}, though while working on the higher level of the derived categories.
\subsection{Generators of the Heisenberg algebra}\label{Heisgen}
Let $V$ be a vector space  
over $\mathbb Q$ (not necessarily  of finite dimension) with a bilinear form $\langle\_,\_\rangle$. Set $L_V= V^{\Z\setminus \{0\}}$ and denote the image of $\beta \in V$ in the $n$-th factor of $L_V$ by $a_\beta(n)$ for $n\in \Z\setminus \{0\}$. The \textit{Heisenberg algebra $\Heis_V$} associated to $(V,\langle \_,\_\rangle)$ is the unital $\Q$-algebra given by the tensor algebra $T(L_V)$ modulo the relations  
\begin{align}\label{originalrel}
 [a_\alpha(m),a_\beta(n)]=\delta_{m,-n} n\langle \alpha,\beta\rangle \,.
\end{align}
Note that, with our definition, the Heisenberg algebra is the quotient of the universal enveloping algebra $U(\mathfrak h_V)$ of the Heisenberg Lie algebra associated to $(V,\langle\_,\_\rangle)$ where the central charge is identified with 1. Hence, every module over the Heisenberg algebra $\Heis_V$ yields a representation of the Heisenberg Lie algebra $\mathfrak h_V$. In particular, the Fock space representation is induced from a $\Heis_V$-module; see Section \ref{Fock}. 

We define elements $p^{(n)}_\beta$ and $q^{(n)}_\beta$ for $\beta\in V$ and $n$ a non-negative integer by the formulae
\begin{align}\label{pqdefin}
\sum_{n\ge 0}p^{(n)}_\beta z^n=\exp\left( \sum_{\ell \ge 1} \frac{a_\beta(-\ell)}{\ell} z^\ell\right)\quad,\quad \sum_{n\ge 0}q^{(n)}_\beta z^n=\exp\left(\sum_{\ell \ge 1} \frac{a_\beta(\ell)}{\ell} z^\ell\right)\,.
\end{align}
As $\beta$ runs through a basis of $V$ and $n$ through the positive integers, the elements $q_\beta^{(n)}$ and $p_\alpha^{(n)}$ together form a set of generators of $\Heis_V$. In order to describe their relations, we introduce the following notation. 
\begin{defin}
For $\chi\in \mathbb Q$ and $k$ a non-negative integer, we set
\[
s^k\chi:=\binom{\chi +k-1}{k}:=\frac 1{k!}(\chi+k-1)(\chi+k-2)\cdots (\chi+1)\chi\,.
\]
\end{defin}
\begin{lemma}\label{relationlem}
The relations among the above generators are given by 
\begin{align}
&[q^{(m)}_\alpha,q^{(n)}_\beta]=0=[p^{(m)}_\alpha,p^{(n)}_\beta]\,,\label{trivrel}\\
&q^{(m)}_\alpha p^{(n)}_\beta=\sum_{k=0}^{\min\{m,n\}} s^k\langle \alpha,\beta\rangle \cdot p^{(n-k)}_\beta q^{(m-k)}_\alpha\,. \label{nontrivrel}
\end{align}
\end{lemma}
We will prove Lemma \ref{relationlem} in Section \ref{lemproofsect}.
Note the following simple but important fact.
\begin{lemma}\label{symmetricEuler}
Let $W^*$ be a finite-dimensional graded vector space. Then the Euler characteristic of its symmetric product $S^kW^*$ (formed in the graded sense) is given by
\begin{align*}
 \chi(S^kW^*)=s^k(\chi(W^*))\,.
\end{align*}    
\end{lemma}

\subsection{Construction and results}\label{constr}
Let $X$ be a smooth complete variety over a field $\Ik$ of characteristic zero. For a non-negative integer $\ell$, we consider the cartesian product $X^\ell$ together with the natural action of the symmetric group $\sym_\ell$ given by permuting the factors. The associated quotient stack $[X^\ell/\sym_\ell]$ is called the \textit{$\ell$-th symmetric quotient stack}. Its derived category can equivalently be described as the $\sym_\ell$-equivariant derived category of the cartesian product, that means $\D([X^\ell/\sym_\ell])\cong \D_{\sym_\ell}(X^\ell)$. We set 
\begin{align*}
 \ID:= \bigoplus_{\ell\ge 0} \D_{\sym_\ell}(X^\ell)
\end{align*} 
For $1\le n\le N$ and $\beta\in \D(X)$ we define the functor 
\begin{align}\label{Pdefin}
P^{(n)}_{N,\beta}\colon \D_{\sym_{N-n}}(X^{N-n})\to \D_{\sym_{N}}(X^{N})\quad,\quad E\mapsto \Inf_{\sym_n\times \sym_{N-n}}^{\sym_N}\bigl(\beta^{\boxtimes n}\boxtimes E  \bigr)\,.  
\end{align}
Note that for $E\in \D_{\sym_{N-n}}(X^{N-n})$ we can consider $\beta^{\boxtimes n}\boxtimes E$ canonically as an object of $\D_{\sym_n\times \sym_{N-n}}(X^N)$.
The \textit{inflation} functor $\Inf_{\sym_n\times \sym_{N-n}}^{\sym_N}\colon \D_{\sym_n\times \sym_{N-n}}(X^N)\to\D_{\sym_N}(X^N)$ is the adjoint of the forgetful functor; see Section \ref{equibasics} for details on the functor $\Inf$ and Section \ref{PQdesc} for details on the functor $P^{(n)}_{N,\beta}$. We set 
\begin{align*}
P^{(n)}_\beta:=\bigoplus_{N\ge n} P^{(n)}_{N,\beta}\colon \ID\to \ID \quad\text{for $n\ge 1$,}\quad P^{(0)}_\beta:=\id\colon \ID\to \ID\,. 
\end{align*}
Finally, we define $Q^{(n)}_\beta\colon \ID\to \ID$ as the right adjoint of $P^{(n)}_\beta$.  
\begin{theorem}\label{mainthm}
For every $\alpha,\beta\in \D(X)$ and $n,m\in \N$, we have the relations
\begin{align}
& Q^{(m)}_\alpha Q^{(n)}_\beta\cong Q^{(n)}_\beta Q^{(m)}_\alpha\quad,\quad P^{(m)}_\alpha P^{(n)}_\beta\cong P^{(n)}_\beta P^{(m)}_\alpha \label{trivialfunctorrel} \\
&  Q^{(m)}_\alpha P^{(n)}_\beta\cong \bigoplus_{k=0}^{\min\{m,n\}} S^k\Hom^*(\alpha,\beta)\otimes_\Ik P^{(n-k)}_\beta Q^{(m-k)}_\alpha\,.\label{nontrivialfunctorrel}   
\end{align}
\end{theorem}
We denote by $\K([X^n/\sym_n])\cong \K_{\sym_n}(X^n)$ the equivariant Grothendieck group with coefficients in $\Q$ and set $\mathbb K:=\bigoplus_{\ell\ge 0}\K_{\sym_n}(X^n)=\K(\mathbb D)$. We consider $\K(X)$ together with the bilinear form $\langle\_,\_\rangle$ given by the Mukai pairing, i.e.\ 
\begin{align}\label{bidefin}
\langle [\alpha],[\beta]\rangle :=\chi(\alpha, \beta):=\chi\bigl(\Hom^*(\alpha,\beta)  \bigr) 
\end{align}
where $[\alpha]$ denotes the class of an object $\alpha\in \D(X)$ in the Grothendieck group $\K(X)$
\begin{cor}
 The descent of the functors $P^{(n)}_\beta$ and $Q^{(n)}_\beta$ to the level of the Grothendieck groups makes $\mathbb K$ into a representation of the Heisenberg algebra $\Heis_{\K(X)}$.
\end{cor}
\begin{proof}
 Let $p^{(n)}_\beta:=[P^{(n)}_\beta]\colon \mathbb K\to \mathbb K$ and $q^{(n)}_\beta:=[Q^{(n)}_\beta]\colon \mathbb K\to \mathbb K$ be the induced maps on the level of the Grothendieck groups. By Lemma \ref{symmetricEuler} and (\ref{bidefin}), the relations of Theorem \ref{mainthm} translate to the $p^{(n)}_\beta$ and $q^{(n)}_\beta$ satisfying the relations of Lemma \ref{relationlem}.  
\end{proof}

The category $\mathbb D$ together with the functors $Q_\beta^{(n)}$ and $P_\beta^{(n)}$ can be regarded as a categorification of the Heisenberg representation $\mathbb K$. 

Since all the $q^{(n)}_\beta$ vanish on $\mathbb K_0=\K_{\sym_0}(X^0)= \K(\text{point})\cong \mathbb Q$, one gets an embedding of the Fock space representation into $\mathbb K$; see Section \ref{Fock} for details.

We will give the proof of Theorem \ref{mainthm} in Sections \ref{mainproof1} and \ref{mainproof2}.
One main ingredient is the following easy fact.

\begin{lemma}[``Symmetric K\"unneth formula'']\label{symmetricKun}
Let $\alpha,\beta\in \D(X)$ and $k\in \N$. Then
\[
 \Hom^*_{\D_{\sym_k}(X^k)}(\beta^{\boxtimes k},\alpha^{\boxtimes k})\cong S^k \Hom^*(\beta,\alpha)
\]
\end{lemma}
\subsection{Conclusion}
The reason for the occurrence of Heisenberg actions in the context of  symmetric quotient stacks, and hence also of Hilbert schemes of points on surfaces, can be summarised in the following simple and, in the author's opinion, satisfying way:
\begin{enumerate}
 \item The Heisenberg algebra has generators whose relations involve the numbers $s^n\langle \alpha,\beta\rangle$ as coefficients; see Lemma \ref{relationlem}.
 \item   A ``vector-spaceification'' of these numbers is given by the graded vector spaces $S^n\Hom^*(\alpha,\beta)$; see Lemma \ref{symmetricEuler}.
 \item Because of the symmetric K\"unneth formula (Lemma \ref{symmetricKun}), these graded vector spaces show up very naturally in the context of derived categories of symmetric quotient stacks. In particular, the relatively simple construction (\ref{Pdefin}) yields a categorical action of the Heisenberg algebra.
\end{enumerate}
\subsection{Organisation of the paper}
In Section \ref{lemproofsect} we prove Lemma \ref{relationlem}. Afterwards, we recall some basic facts on equivariant derived categories and functors in Section \ref{equibasics}, introduce some notation in Section \ref{notations}, and describe the functors $P^{(n)}_\beta$ and $Q^{(n)}_\beta$ in some more detail in Section \ref{PQdesc}. In Sections \ref{mainproof1} and \ref{mainproof2}, we give the proof of Theorem \ref{mainthm}. 

We explain in Section \ref{Fock} the embedding of the Fock space representation into $\mathbb K$. In Section \ref{transposed} we define further functors which give a categorification of the action on $\mathbb K$ with respect to the so-called transposed generators of $\Heis_{\K(X)}$. In particular, this recovers the construction of \cite{Khov}.
In Section \ref{gen} we discuss some generalisations and variants, for example to the non-complete case and the case where we replace the category $\D(X)$ by some equivariant category. Finally, we point out two open problems in Section \ref{open}. 
\smallskip
\textbf{Conventions}
With the exception of Section \ref{gen}, $X$ will be a smooth complete variety over a field $\Ik$ of characteristic zero. 
Its \textit{derived category} is understood to be the bounded derived category of coherent sheaves, i.e. $\D(X):=\D^b(\Coh(X))$. All functors between derived categories are assumed to be derived although
this will not be reflected in the notation.  
\smallskip
\textbf{Acknowledgements.} The author was financially supported by the research grant KR 4541/1-1 of the DFG. He thanks Daniel Huybrechts, Ciaran Meachan, David Ploog, Miles Reid, and Pawel Sosna for helpful comments.
\section{Proofs}
\subsection{Proof of Lemma \ref{relationlem}}\label{lemproofsect}
Let $\alpha, \beta \in V$ and $m,n\in \N$. The vanishing of the commutators  
$[q^{(m)}_\alpha,q^{(n)}_\beta]$ and $[q^{(m)}_\alpha,q^{(n)}_\beta]$ follows immediately from the fact that $[a_\alpha(r),a_\beta(s)]=0$ if $r$ and $s$ are both positive or both negative.

For the proof of relation (\ref{nontrivrel}) we follow closely the proof of \cite[Lem.\ 1]{CL}. Set
\[ A(z):=\sum_{\ell \ge 1} \frac{a_\beta(-\ell)}{\ell} z^\ell\quad,\quad B(w):=\sum_{\ell \ge 1} \frac{a_\alpha(\ell)}{\ell} w^\ell
\]
so that $q^{(m)}_\alpha p^{(n)}_\beta=[w^mz^n]\exp(B)\exp(A)$. 
We also set $\chi:=\langle \alpha,\beta\rangle$. 
By relation (\ref{originalrel}), we have 
\begin{align}\label{a}
 [B,A]=\sum_{\ell\ge 1} \frac{\ell\cdot\chi}{\ell^2}w^\ell z^{\ell} =\chi\sum_{\ell\ge 1} \frac{(w z)^\ell}{\ell}=-\chi\log(1-wz)\,.  
\end{align}
In particular, $[B,A]$ commutes with $A$ and $B$. Thus, 
\begin{align}\label{b}
 \exp(B)\exp(A)=\exp([B,A])\exp(A)\exp(B)\,;
\end{align}
see \cite[Lem.\ 9.43]{Nakbook}. Recall the binomial formula
\begin{align}\label{c}
 (1-t)^{-\chi}=\sum_{k\ge 0} s^k\chi\cdot t^k\,.
\end{align}
Together, (\ref{a}), (\ref{b}), and (\ref{c}) give
\[
 \exp(B)\exp(A)=(1-wz)^{-\chi}\exp(A)\exp(B)=\left( \sum_{k\ge 0} s^k\chi\cdot (wz)^k\right) \exp(A)\exp(B)\,.
\]
Comparing the coefficients of $w^mz^n$ on both sides gives relation (\ref{nontrivrel}). 

In order to show that there are no further relations among the generators $q_\beta^{(n)}$ and $p_\beta^{(n)}$ we need to introduce the following notation. 
Let $\{\beta_{i}\}_{i\in I}$ be a basis of $V$. 
Consider maps
\begin{align}\label{map}\nu\colon I\to \coprod_{n\ge 0}\bigl\{\text{partitions of $n$}  \bigr\}\end{align}
such that $\nu(i)\neq 0$ only for a finite number of $i\in I$. We write the partitions in the form $\nu(i)=1^{\nu(i)_1} 2^{\nu(i)_2}\cdots$. We fix a total order on $I$  and set 
\begin{align*}a(\nu):=\prod_{i\in I}\bigl(\prod_{k\ge 1} a_{\beta_i}(k)^{\nu(i)_k}\bigr)\quad&,\quad a(-\nu):=\prod_{i\in I}\bigl(\prod_{k\ge 1} a_{\beta_i}(-k)^{\nu(i)_k}\bigr)\,,\\
q(\nu):=\prod_{i\in I}\bigl(\prod_{k\ge 1} (q_{\beta_i}^{(k)})^{\nu(i)_k}\bigr)\quad&,\quad p(\nu):=\prod_{i\in I}\bigl(\prod_{k\ge 1} (p_{\beta_i}^{(k)})^{\nu(i)_k}\bigr)\,.
\end{align*}
Here, the inner product is formed with respect to the usual order of $\N$ and the outer product with respect to the fixed order of $I$.
Since the only relation between the generators $a_\beta(k)$ is the commutator relation (\ref{originalrel}), the elements $a(-\nu)a(\mu)$ form a basis of $\Heis_V$ as a vector space. In the algebra with generators $q_{\beta_i}^{(n)}$ and $p_{\beta_i}^{(n)}$ and the relations (\ref{trivrel}) and (\ref{nontrivrel}), one can write every element as a linear combination of the form
\begin{align}\label{lincomb}
 \sum_{\nu,\mu} \lambda(\nu,\mu)p(\nu)q(\mu)\quad,\quad \lambda(\mu,\nu)\in \mathbb Q\,.
\end{align}
Thus, it is sufficient to show that the $p(\nu)q(\mu)$ are linearly independent in $\Heis_V$. 
For two maps $\nu,\nu'$ as in (\ref{map}) we say that $\nu$ is coarser than $\nu'$, and write $\nu\succ \nu'$, if $|\nu(i)|=|\nu'(i)|$ and $\nu(i)$ is a coarser partition     
than $\nu'(i)$ for every $i\in I$. For pairs, we set $(\nu,\mu)\succ(\nu',\mu')$ if $\nu\succ\nu'$ and $\mu\succ \mu'$. Note that by (\ref{pqdefin}), we have
\begin{align*}
p^{(n)}_\beta=a_\beta(-n)+ \bigl(\text{linear combination of $a_\beta(-1)^{\nu_1}a_\beta(-2)^{\nu_2}\cdots$ for partitions $\nu$ of $n$}   \bigr)
\end{align*}
and similarly for $q^{(n)}_\beta$. Thus, 
\begin{align}\label{aa}
p(\nu)q(\mu)=a(-\nu)a(\mu)+ \bigl(\text{linear combination of $a(-\nu')a(\mu')$ for $(\nu,\mu)\succ(\nu',\mu')$}   \bigr)\,. 
\end{align}
Let $x= \sum_{\nu,\mu} \lambda(\mu,\nu)p(\nu)q(\mu)$ as in 
(\ref{lincomb}) such that not all coefficients vanish. Pick a pair $(\nu_0,\mu_0)$ with $\lambda(\nu_0,\mu_0)\neq 0$ such that for each coarser pair the corresponding coefficient vanishes. Then, by (\ref{aa}), $x$ is a linear combination of the $a(-\nu)b(-\mu)$ such that the coefficient of $a(-\nu_0)a(\mu_0)$ equals $\lambda(\nu_0,\mu_0)$. Since the $a(-\nu)b(\mu)$ form a basis, $x\neq 0$ in $\Heis_V$. 
\subsection{Basic facts about equivariant functors}\label{equibasics}
For details on equivariant derived categories and functors between them, we refer to 
\cite[Sect.\ 4]{BKR}.
Let $G$ be a finite group acting on a smooth projective variety $M$. 
The \textit{equivariant derived category} $\D_G(M)$ has as objects pairs $(E,\lambda)$ where $E\in \D(M)$ and $\lambda$ is a $G$-linearisation of $E$. A \textit{$G$-linearisation} is a family of isomorphisms $(E\xrightarrow\cong g^*E)_{g\in G}$ in $\D(M)$ such that for every pair $g,h\in G$ the composition 
\[
 E\xrightarrow{\lambda_g}g^*E\xrightarrow{g^*\lambda_h}g^*h^*E\cong (hg)^*E
\]
equals $\lambda_{hg}$. The morphisms are morphisms in the ordinary derived category $\D(X)$ which are compatible with the linearisations. The category $\D_G(M)$ is equivalent to the bounded derived category $\D^b(\Coh_G(M))$ of $G$-equivariant coherent sheaves on $M$; see \cite{Chen} or \cite{Elagin}. 

For a subgroup $U\subset G$ there is the functor $\Res_G^U\colon \D_G(M)\to \D_U(M)$ given by restricting the linearisations. It has the \textit{inflation} functor $\Inf_U^G\colon \D_U(M)\to \D_G(M)$ as a left and right adjoint.
Let $U\setminus G$ be the left cosets and $E\in \D(X)$. Then 
\begin{align}\label{Infdef}
\Inf_U^G(E)\cong \bigoplus_{\sigma\in U\setminus G} \sigma^*E 
\end{align}
with the $G$-linearisation of $\Inf_U^G(E)$ given by a combination of the $U$-linearisation of $E$ and permutation of the direct summands.

Let $G$ act trivially on $M$. Then there is the functor $\triv\colon \D(M)\to \D_G(M)$ which equips every object with the trivial linearisation. It has the functor of invariants $(\_)^G\colon \D_G(M)\to \D(M)$ as a left and right adjoint. 

We use the following principle for the computation of invariants; see e.g.\ \cite[Sect.\ 3.5]{Kru4}.
Let $\cE=(E,\lambda)$ be a $G$-equivariant sheaf such that $E=\oplus_{i\in \cK}E_i$ for some finite index set $\cK$. Assume that there is an action of $G$ on $\cK$ such that $\lambda_g(E_i)=E_{g\cdot i}$ for all $i\in \cK$ and $g\in G$. We say that the linearisation $\lambda$ induces the action on the index set $\cK$. 
\begin{lemma}\label{Danlemma}
Let $\cR\subset \cK$ be a set of representatives of the $G$-orbits in $\cK$ under the action induced by the linearisation. Then $E^G\cong \bigoplus_{i\in \cR} E_i^{G_i}$ where $G_i\subset G$ is the stabiliser subgroup.
\end{lemma}
\subsection{Combinatorial notations}\label{notations}
For a finite set $I$ we set $X^I:=\prod_{I}X\cong X^{|I|}$. For $J\subset I$ we denote by $\pr_J^I\colon X^I\to X^J$ the projection. 
We often drop the index $I$ in the notation and simply write $\pr_J$ when the source of the projection should be clear from the context.

For $n\in \N$, we set $\n:=[1,n]=\{1,2,\dots,n\}$.
By convention, $\underline 0:=\emptyset$.
Later, there will be a fixed number $N\in \N$ such that all occurring finite sets will be subsets of $\underline N$. Hence, we will often write $N$ instead of $\underline N$ to ease the notation. Furthermore,
for a subset $J\subset \underline N$ we set $\bar J:=\underline N\setminus J$. For $0\le n\le N$ we set $\ncomp:=\overline{\n}=\underline N\setminus \n=[n+1,N]$. 
\subsection{The functors $P^{(n)}_\beta$ and $Q^{(n)}_\alpha$}\label{PQdesc}
For every object $\beta\in \D(X)$ and $N\ge n$, the functor $P^{(n)}_{N,\beta}\colon \D_{\sym_{N-n}}(X^{N-n})\to \D_{\sym_N}(X^N)$ is given by the composition
\begin{align*}
&\D_{\sym_{N-n}}(X^{N-n})\xrightarrow{\triv} \D_{\sym_{n}\times \sym_{N-n}}(X^{N-n})\xrightarrow{\pr_{\ncomp}^{*}}\D_{\sym_{n}\times \sym_{N-n}}(X^{N})\\
&\xrightarrow{\otimes\pr_{\n}^{*}\beta^{\boxtimes n}} \D_{\sym_{n}\times \sym_{N-n}}(X^{N})\xrightarrow{\Inf_{\sym_{\n}\times \sym_{\ncomp}}^{\sym_N}} \D_{\sym_{N}}(X^{N}) 
\end{align*}
where we use the identifications $X^{\ncomp}\cong X^n$ and $X^{\ncomp}=X^{[n+1,N]}\cong X^{N-n}$. Hence, the right-adjoint $Q^{(n)}_{N,\beta}\colon \D_{\sym_{N}}(X^{N})\to \D_{\sym_{N-n}}(X^{N-n})$ is the composition 
\begin{align*}
&\D_{\sym_{N-n}}(X^{N-n})\xleftarrow{(\_)^{\sym_n}} \D_{\sym_{n}\times \sym_{N-n}}(X^{N-n})\xleftarrow{\pr_{\ncomp*}}\D_{\sym_{n}\times \sym_{N-n}}(X^{N})\\
&\xleftarrow{\otimes\pr_{\n}^{*}\beta^{\vee \boxtimes n}} \D_{\sym_{n}\times \sym_{N-n}}(X^{N})\xleftarrow{\Res_{\sym_{\n}\times \sym_{\ncomp}}^{\sym_N}} \D_{\sym_{N}}(X^{N})\,. 
\end{align*}
Note that $(\beta^\vee)^{\boxtimes n}\cong (\beta^{\boxtimes n})^\vee$ so that we simply write $\beta^{\vee \boxtimes n}$ without ambiguity. Since  
\begin{align*}
\sigma^*\bigl(\pr_{\n}^*\beta^\boxtimes\otimes \pr_{\ncomp }^* E\bigr)\cong \pr_{\sigma^{-1}(\n)}^*\beta^\boxtimes\otimes \pr_{\sigma^{-1}(\ncomp )}^* E\bigr) 
\end{align*}
for every $E\in \D_{\sym_{N-n}}(X^{N-n})$ and $\sigma\in \sym_n$, we get by (\ref{Infdef})
\begin{align}\label{Pformula}
P^{(n)}_{N,\beta}(E)\cong\bigoplus_{I\subset \underline N, |I|=n} \pr_I^*\beta^{\boxtimes n}\otimes \pr_{\bar I}^*E
\end{align}
Furthermore, the functor $Q^{(m)}_\alpha$ is given on objects $F\in \D_{\sym_N}(X^N)$ by
\begin{align}\label{Qformula}
 Q^{(m)}_\alpha(F)\cong \pr_{\mcomp*}(\pr_{\m}^* \alpha^{\vee\boxtimes m}\otimes F)^{\sym_{\m}}
\end{align}
\begin{remark}\label{flexible}
Let $I$ be any set of cardinality $N$ and $J\subset I$ be any subset of cardinality $m$. Making the identifications $X^I\cong X^N$, $X^J\cong X^m$, and $X^{I\setminus J}\cong X^{N-m}$ we can also write
\begin{align*}
 Q^{(m)}_{N,\alpha}(F)\cong \pr_{I\setminus J*}^I(\pr_{J}^{I*} \alpha^{\vee\boxtimes m}\otimes F)^{\sym_{J}}\,.
\end{align*}
\end{remark}
\subsection{Proof of relation (\ref{trivialfunctorrel})}\label{mainproof1}
In order to verify the relations (\ref{trivialfunctorrel}) of Theorem \ref{mainthm} it is sufficient to show the first relation, i.e. $Q^{(m)}_\alpha Q^{(n)}_\beta\cong Q^{(n)}_\beta Q^{(m)}_\alpha$ for $\alpha,\beta\in \D(X)$ and $m,n\in \N$. 
The relation $P^{(m)}_\alpha P^{(n)}_\beta\cong P^{(n)}_\beta P^{(m)}_\alpha$ follows by adjunction.
This means that we have to show that 
\[
 Q^{(m)}_{N-n,\alpha} Q^{(n)}_{N,\beta}\cong Q^{(n)}_{N-m,\beta} Q^{(m)}_{N,\alpha}\,\colon \D_{\sym_N}(X^N)\to \D_{\sym_{N-m-n}}(X^{N-m-n}) \quad\text{for every $N\ge n+m$.}
\]
Indeed, using Remark \ref{flexible}, the projection formula along  $\pr^N_{\ncomp}$, the projection formula along $\pr_{\overline{[n+1,n+m]}}^N$, and finally Remark \ref{flexible} again, we get isomorphisms 
\begin{align*}
& Q^{(m)}_{N-n,\alpha} Q^{(n)}_{N,\beta}(F)\\
\cong&
\Bigl[\pr_{\mncomp*}^{\ncomp}\Bigl(\pr_{[n+1,n+m]}^{\ncomp*}\alpha^{\vee\boxtimes m}\otimes \bigl[\pr_{\ncomp*}^{N}(\pr_{\n}^{N*}\beta^{\vee\boxtimes n}\otimes F  )  \bigr]^{\sym_{\n}}  \Bigr)\Bigr]^{\sym_{[n+1,n+m]}}\\
\cong&
\Bigl[\pr_{\mncomp*}^{N}\bigl(\pr_{\n}^{N*}\beta^{\vee\boxtimes n}\otimes \pr_{[n+1,n+m]}^{N*}\alpha^{\vee\boxtimes m}\otimes F     \bigr)\Bigr]^{\sym_{\n}\times \sym_{[n+1,n+m]}} 
\\
\cong&
\Bigl[\pr_{\mncomp*}^{\overline{[n+1,n+m]}*}\Bigl(\pr_{\n}^{\overline{[n+1,n+m]}*}\beta^{\vee\boxtimes n}\otimes \bigl[\pr_{\overline{[n+1,n+m]}*}^{N}(\pr_{[n+1,n+m]}^{N*}\alpha^{\vee\boxtimes m}\otimes F  )  \bigr]^{\sym_{[n+1,n+m]}}  \Bigr)\Bigr]^{\sym_{\n}}\\
\cong& Q^{(n)}_{N-m,\beta} Q^{(m)}_{N,\alpha}(F)
\end{align*}
which are functorial in $F\in \D_{\sym_N}(X^N)$.
\subsection{Proof of relation (\ref{nontrivialfunctorrel})}\label{mainproof2}
Let $N\ge \max\{m,n\}$ and $E\in \D_{\sym_{N-n}}(X^{N-n})$. 
Combining (\ref{Pformula}) and (\ref{Qformula}) we get  
\begin{align*}
Q^{(m)}_\alpha P^{(n)}_\beta (E)\cong \bigl[\bigoplus_{I\in \cK} T_I    \bigr]^{\sym_{\m}}
\end{align*}
where
\begin{align*}
\mathcal K=\{I\subset \underline N\,,\, |I|=n\}\quad\text{and}\quad T_I=\pr_{\mcomp*}\bigl(\pr_{\m}^*\alpha^{\vee\boxtimes m}\otimes \pr_I^*\beta^{\boxtimes n}\otimes \pr_{\bar I}^*E   \bigr)\,. 
\end{align*}
The $\sym_{\m}$-linearisation of $\otimes_{I\in \cK} T_I$ induces on the index set $\mathcal K$ the action $\sigma\cdot I=\sigma(I)$. A set of representatives of the $\sym_{\m}$-orbits in $\mathcal K$ is given by
\[\cR=\{I=\uk\cup J\mid n+m-N\le k\le \min\{m,n\}\,,\, J\subset \mcomp \,,\, |J|=n-k\}\,.\]
Thus, Lemma \ref{Danlemma} yields
\begin{align}\label{decompinva}
 Q^{(m)}_\alpha P^{(n)}_\beta (E)\cong \bigoplus_{I=\uk\cup J\in \cR} T_I^{\sym_{\uk}\times \sym_{[k+1,m]}}\,.
\end{align}
Let $I=\uk\cup J\in \cR$ and consider the diagram
\begin{align}\label{largediag}
 \xymatrix{
  &  X^{\mcomp} \ar_{\pr_{\mcomp\setminus J}^{\mcomp}}[dl] \ar^{\pr_{J}^{\mcomp}}[rr]   &  &  X^{\mcomp\cap I}=X^J  \\
 X^{\Icomp\cap\mcomp}=X^{\mcomp\setminus J}  & X^{\kcomp}=X^{\mcomp\cup \Icomp} \ar^{\pr_{\mcomp}^{\kcomp}}[u]\ar_{\pr_{\Icomp}^{\kcomp}}[d]  & X^N \ar_{\pr_{\mcomp}^{N}}[ul] \ar^{\quad\pr_{\kcomp}^{N}}[l] \ar^{\pr_{\Icomp}^{N}}[dl] \ar^{\pr_{\uk}^{N}\quad}[r]  & X^{\m\cap I}=X^{\uk}  \\
  & X^{\Icomp} \ar^{\pr_{\mcomp\setminus J}^{\Icomp}}[ul] \ar^{\pr_{[k+1,m]}^{\Icomp}}[rr] &  & X^{\m\cap \Icomp}=X^{[k+1,m]}\,.
}
\end{align}
The two triangles in the middle of the diagram are commutative and the square on the left is cartesian. We can rewrite $T_I=T_{\uk\cup J}$ as
\begin{align*}
 T_I\cong \pr_{\mcomp *}^{\kcomp} \pr_{\kcomp *}^{N} \left[\pr_{\kcomp}^{N*}\Bigl(\pr_{\mcomp}^{\kcomp*}\pr_{J}^{\mcomp*}
\beta^{\boxtimes n-k}\otimes \pr_{\Icomp}^{\kcomp*}\bigl(E\otimes \pr_{[k+1,m]}^{\Icomp *}\alpha^{\vee\boxtimes m-k}\bigr)   \Bigr)    
\otimes \pr_{\uk}^{N*}(\alpha^\vee\otimes \beta)^{\boxtimes k} \right]\,.
\end{align*}
By the projection formula along $\pr_{\kcomp}^N$, we get
\begin{align*}
 T_I\cong \pr_{\mcomp *}^{\kcomp}  \left[\pr_{\mcomp}^{\kcomp*}\pr_{J}^{\mcomp*}
\beta^{\boxtimes n-k}\otimes \pr_{\Icomp}^{\kcomp*}\bigl(E\otimes \pr_{[k+1,m]}^{\Icomp *}\alpha^{\vee\boxtimes m-k}\bigr)       
\otimes \pr_{\kcomp* }^{N}\pr_{\uk}^{N*}(\alpha^\vee\otimes \beta)^{\boxtimes k} \right]\,.
\end{align*}
By the fact that $X^N= X^{\uk}\times X^{\kcomp}$ and K\"unneth formula
\begin{align*}
\pr_{\kcomp* }^{N}\pr_{\uk}^{N*}(\alpha^\vee\otimes \beta)^{\boxtimes k}\cong \reg_{X^{\kcomp}}\otimes_{\Ik} \glob\bigl((\alpha^\vee\otimes \beta)^{\boxtimes k}\bigr)&\cong  \reg_{X^{\kcomp}}\otimes_{\Ik} \Hom^*(\alpha^{\boxtimes k},\beta^{\boxtimes k})\\&\cong
\reg_{X^{\kcomp}}\otimes_{\Ik} \Hom^*(\alpha,\beta)^{\otimes k}\,. 
\end{align*}
This gives 
\begin{align*}
 T_I\cong\Hom^*(\alpha,\beta)^{\otimes k}\otimes_\Ik \tilde T_I\quad, \quad \tilde T_I=\pr_{\mcomp *}^{\kcomp}  \left[\pr_{\mcomp}^{\kcomp*}\pr_{J}^{\mcomp*}
\beta^{\boxtimes n-k}\otimes \pr_{\Icomp}^{\kcomp*}\bigl(E\otimes \pr_{[k+1,m]}^{\Icomp *}\alpha^{\vee\boxtimes m-k}\bigr)       
 \right]\,.
\end{align*}
By the projection formula along $\pr_{\mcomp}^{\kcomp}$ and base change along the cartesian square in (\ref{largediag}), 
\begin{align*}
\tilde T_I&\cong   \pr_{J}^{\mcomp*}
\beta^{\boxtimes n-k}\otimes \pr_{\mcomp*}^{\kcomp}\pr_{\Icomp}^{\kcomp*}\bigl(E\otimes \pr_{[k+1,m]}^{\Icomp *}\alpha^{\vee\boxtimes m-k}\bigr)\\
&\cong 
  \pr_{J}^{\mcomp*}
\beta^{\boxtimes n-k}\otimes \pr_{\mcomp\setminus J}^{\mcomp*}\pr_{\mcomp\setminus J*}^{\Icomp}\bigl(E\otimes \pr_{[k+1,m]}^{\Icomp *}\alpha^{\vee\boxtimes m-k}\bigr)
\end{align*}
In summary, 
\begin{align*}
T_I\cong\Hom^*(\alpha,\beta)^{\otimes k}\otimes_\Ik \pr_{J}^{\mcomp*}
\beta^{\boxtimes n-k}\otimes \pr_{\mcomp\setminus J}^{\mcomp*}\pr_{\mcomp\setminus J*}^{\Icomp}\bigl(E\otimes \pr_{[k+1,m]}^{\Icomp *}\alpha^{\vee\boxtimes m-k}\bigr)\,. 
\end{align*}
Taking $\sym_{\uk}\times \sym_{[k+1,m]}$-invariants yields
\begin{align*}
T_I^{\sym_{\uk}\times \sym_{[k+1,m]}}\cong S^k\Hom^*(\alpha,\beta)\otimes_\Ik \pr_{J}^{\mcomp*}
\beta^{\boxtimes n-k}\otimes \pr_{\mcomp\setminus J}^{\mcomp*}\Bigl[\pr_{\mcomp\setminus J*}^{\Icomp}\bigl(E\otimes \pr_{[k+1,m]}^{\Icomp *}\alpha^{\vee\boxtimes m-k}\bigr)\Bigr]^{\sym_{[k+1,m]}}. 
\end{align*}
Plugging this into (\ref{decompinva}) gives
\begin{align*}
 &Q_\alpha^{(m)}P_\beta^{(n)}(E) \cong \bigoplus_{k=n+m-N}^{\min\{m,n\}}S^k\Hom^*(\alpha,\beta)\otimes \widehat T_k
\end{align*}
where 
\[\widehat T_k:=\bigoplus_{J\subset \mcomp,\, |J|=n-k}\pr_{J}^{\mcomp*}
\beta^{\boxtimes n-k}\otimes \pr_{\mcomp\setminus J}^{\mcomp*}\bigl[\pr_{\mcomp\setminus J*}^{\Icomp}(E\otimes \pr_{[k+1,m]}^{\Icomp *}\alpha^{\vee\boxtimes m-k})\bigr]^{\sym_{[k+1,m]}}\,.\]
Using Remark \ref{flexible}, 
we see that $\widehat T_k\cong P^{(n-k)}_{N-m,\beta} Q^{(m-k)}_{N-n,\alpha}(E)$. Hence,
\begin{align*}
 Q_\alpha^{(m)}P_\beta^{(n)}(E) \cong \bigoplus_{k=n+m-N}^{\min\{m,n\}}S^k\Hom^*(\alpha,\beta)\otimes_\Ik P^{(n-k)}_{N-m,\beta} Q^{(m-k)}_{N-n,\alpha}(E)\,.
\end{align*}
Since all the isomorphisms used above are functorial, this shows the isomorphism (\ref{nontrivialfunctorrel}).
\section{Further remarks}
\subsection{The Fock space}\label{Fock}
The \textit{Fock space} representation of the Heisenberg algebra is defined as the quotient $\Fock_V:=\Heis_V/I$ where $I$ is the left ideal generated by all the $a_\beta(n)$ with $n>0$. 
Let $\mathbf 1\in \Fock_V$ be the class of $1\in \Heis_V$. 
The Fock space is an irreducible representation. Hence, given any representation $M$ of $\Heis_V$ together with an element $0\neq x\in M$ such that $a_\beta(n)\cdot x=0$ for all $\beta\in V$ and $n>0$, the map
\[\Fock_V\to M\quad,\quad \mathbf 1\mapsto x\]
gives an embedding of $\Heis_V$-representations. 

Let $\mathbb K=\bigoplus_{\ell\ge 0} \K_{\sym_\ell}(X^\ell)$ be the representation described in Section \ref{constr}. Then every 
non-zero $x\in \K_{\sym_0}(X^0)\cong \K(\text{point})\cong \mathbb Q$ is annihilated by all the $q^{(n)}_\beta$ and hence by the $a_\beta(n)$ for $n>0$.
Thus, the Fock space can be realised as a subrepresentation of $\mathbb K$.

In general, it is a proper subrepresentation, i.e. $\Fock_{\K(X)}\subsetneq \mathbb K$. However, if $\K(X)$ is of finite dimension and the exterior product $\K(X)^{\otimes n}\to \K(X^n)$ is an isomorphism, for example if $X$ has a cellular decomposition, we do have the equality $\Fock_{\K(X)}= \mathbb K$. In this case
\[
\mathbb K_\ell:= \K_{\sym_\ell}(X^\ell)\cong \bigoplus_{\nu \text{ partition of } \ell} \bigl(S^{\nu_1}\K(X)\bigr)\otimes \bigl(S^{\nu_2} \K(X)\bigr)\otimes \dots
\]
as follows from the general decomposition of equivariant K-theory described in \cite{Vistdecomp}. Thus, the dimensions of $\Fock_{\K(X)}$ and $\mathbb K$ agree in every degree $\ell$.    

\subsection{Transposed generators}\label{transposed}

There is yet another set of generators of the Heisenberg algebra, namely $p^{(1^n)}_{\beta}$, $q^{(1^n)}_\beta$ defined by 
\begin{align*}
\sum_{n\ge 0}p^{(1^n)}_\beta z^n=\exp\left( -\sum_{\ell \ge 1} \frac{a_\beta(-\ell)}{\ell} z^\ell\right)\quad,\quad \sum_{n\ge 0}q^{(1^n)}_\beta z^n=\exp\left(-\sum_{\ell \ge 1} \frac{a_\beta(\ell)}{\ell} z^\ell\right)\,;
\end{align*}
compare \cite[Sect.\ 2.2.2]{CL}. The relations among these generators are exactly the same as those between the $p^{(n)}_{\beta}$ and $q^{(n)}_\beta$. Moreover, one can mix these two sets of generators to get the set of generators consisting of $p^{(n)}_\beta$ and $q^{(1^n)}_\beta$ (or, alternatively, $p^{(1^n)}_\beta$ and $q^{(n)}_\beta$). These are the generators used by Khovanov in \cite{Khov} for his construction of a categorification of the Heisenberg algebra $\Heis_V$ in the case that $V=\Q$. The relations between these generators are 
\begin{align*}
[q^{(1^m)}_\alpha,q^{(1^n)}_\beta]=0=[p^{(m)}_\alpha,p^{(n)}_\beta]\quad,
\quad
q^{(m)}_\alpha p^{(n)}_\beta=\sum_{k=0}^{\min\{m,n\}} s^k(-\langle \alpha,\beta\rangle) \cdot p^{(n-k)}_\beta q^{(m-k)}_\alpha\,. 
\end{align*}
For $\beta\in\D(X)$ and $n>0$, we set $P_\beta^{(1^n)}:=\bigoplus_{N\ge n} P^{(1^n)}_{N,\beta}$ with
\begin{align*}
P^{(1^n)}_{N,\beta}\colon \D_{\sym_{N-n}}(X^{N-n})\to \D_{\sym_{N}}(X^{N})\quad,\quad E\mapsto \Inf_{\sym_n\times \sym_{N-n}}^{\sym_N}\bigl((\beta^{\boxtimes n}\otimes \alt_n)\boxtimes E  \bigr)\,. 
\end{align*}
Here, $\alt_n$ denotes the non-trivial character of $\sym_n$. 
Again, $Q_\beta^{(1^n)}$ is defined as the right-adjoint of $P^{(1^n)}_\beta$.
Using the ``anti-symmetric K\"unneth formula''
\[
 \Hom^*_{\D_{\sym_k}(X^k)}(\beta^{\boxtimes k}\otimes \alt_k,\alpha^{\boxtimes k})\cong \wedge^k \Hom^*(\beta,\alpha)\,,
\]
one can prove in complete analogy to Sections \ref{mainproof1} and \ref{mainproof2} the relations 
\begin{align}
& Q^{(1^m)}_\alpha Q^{(1^n)}_\beta\cong Q^{(1^n)}_\beta Q^{(1^m)}_\alpha\quad,\quad P^{(m)}_\alpha P^{(n)}_\beta\cong P^{(n)}_\beta P^{(m)}_\alpha \label{t1} \\
& Q^{(1^m)}_\beta P^{(n)}_\alpha\cong \bigoplus_{k=0}^{\min\{m,n\}} \wedge^k\Hom^*(\beta,\alpha)\otimes_\Ik P^{(n-k)}_\alpha Q^{(1^{m-k})}_\beta\,. \label{t2}  
\end{align}
One can reformulate Lemma \ref{symmetricEuler} as the formula
$\chi(\wedge^k W^*)=s^k(-\chi(W^*))$ for $W^*$ a graded vector space.
Thus, relations (\ref{t1}) and (\ref{t2}) show that the $P^{(n)}_\beta$ and $Q^{(1^n)}_\beta$ give again a categorical Heisenberg action, this time lifting the generators $p_\beta^{(n)}$ and $q_\beta^{(1^n)}$. The special case that $X$ is a point is exactly the construction of \cite{Khov}.     
\subsection{Generalisations and variants}\label{gen}
Since the construction (\ref{Pdefin}) of the functors $P^{(n)}_\beta$ is very simple and formal, it generalises well to other settings beside smooth complete varieties. 

Let $X$ be a non-complete smooth variety. Then Theorem \ref{mainthm} continues to hold if we assume that the objects $\alpha$ and $\beta$ have complete support, i.e. $\alpha,\beta\in \D^c(X)$, which ensures that $\Hom^*(\alpha,\beta)$ is a finite dimensional graded vector space. 
Thus, we get an induced action of the Heisenberg algebra associated to $\K^c(X):=\K(\D^c(X))$ on $\mathbb K$ (note that $\mathbb K$ needs not necessarily be replaced by Grothendieck groups with finite support).

Similarly, if one wants to also drop the smoothness assumption one needs $\alpha$ and $\beta$ to be perfect objects with complete support.

Furthermore, one can replace the variety $X$ by a Deligne--Mumford stack $\cX$. The case that $\cX=[X/\C^*]$ for $X$ the affine plane or a minimal resolution of a Kleinian singularity gives a Heisenberg action on the equivariant Grothendieck groups 
$\bigoplus_{\ell\ge 0} \K_{\C*}(\Hilb^\ell(X))$ as in \cite{SV}, \cite{FT}, and \cite{CL}.    
\subsection{Open problems}\label{open}
In \cite{GKsym}, the symmetric product $S^n\cT$ of a category $\cT$ is defined in such a way that $S^n(\D(X))\cong \D_{\sym_n}(X^n)$ for a smooth projective variety $X$. Thus, one may hope that our construction can be generalised to a setting where $\D(X)$ is replaced by any (Hom-finite and symmetric monoidal) dg-enhanced triangulated category $\cT$.

Note that in \cite{Khov} and \cite{CL} a categorification of the Heisenberg action is given in a much stronger sense than in this paper. In particular, 2-categories whose Grothendieck groups are isomorphic to the Heisenberg algebra are constructed. Clearly, one would like to do the same in our more general setting too.

\bibliographystyle{alpha}
\bibliography{references}

\end{document}